\documentclass[reqno,a4paper,12pt]{amsart}

\usepackage{amssymb}

\numberwithin{equation}{section}
\date{\today}

\newcommand\x{\mathbf{x}}
\newcommand\y{\mathbf{y}}
\newcommand\ba{\mathbf{\alpha}}
\newcommand\p{\mathbf{p}}
\newcommand\z{\mathbf{z}}

\newcommand{\R}{\mathbb{R}}

\newcommand{\D}{\mathbb{D}}

\newcommand{\Con}{C_0^\infty}

\newcommand{\grad}{\nabla}
\newcommand{\bgrad}{\mathbf{\nabla}}

\newcommand{\B}{{\mathbf{B}}}

\newtheorem{Theorem}{Theorem}
\newtheorem{Corollary}{Corollary}

\begin{document}

\title[Dirac Sobolev inequalities and zero modes]
{Dirac-Sobolev inequalities and estimates for the zero modes of
massless Dirac operators}
\author[A.~Balinsky]{A. Balinsky}
\address{School of Mathematics\\
Cardiff University\\
         23 Senghennydd Road\\
         Cardiff CF2 4YH\\
         UK}
\email{BalinskyA@cardiff.ac.uk}
\author[W.~D.~Evans]{W.~D. Evans}
\address{School of Mathematics\\
Cardiff University\\
         23 Senghennydd Road\\
         Cardiff CF2 4YH\\
         UK}
\email{EvansWD@cardiff.ac.uk}

\author[Y.~Saito]{Y.~Sait\={o}}
\address{Department of Mathematics\\
         University of Alabama at Birmingham\\
         Birmingham, AL 35294-1170\\
         USA}
\email{saito@math.uab.edu}
\thanks { The authors gratefully acknowledge the support of
the EPSRC under grant EP/E04834X/1}
\begin{abstract}
The paper analyses the decay of any zero modes that might exist
for a massless Dirac operator $H:= \ba \cdot (1/i) \bgrad + Q, $
where $Q$ is $4 \times 4$-matrix-valued and of order
$O(|\x|^{-1})$ at infinity. The approach is based on inversion
with respect to the unit sphere in $\R^3$ and establishing
embedding theorems for Dirac-Sobolev spaces of spinors $f$ which
are such that $f$ and $Hf$ lie in $\left(L^p(\R^3)\right)^4, 1\le
p<\infty.$
\end{abstract}
\maketitle

\section{Introduction}

The mathematical interpretation of the stability of matter problem
concerns the question of whether the energy of a system of
particles is bounded from below (stability of the first kind) and
by a constant multiple of the number of particles (stability of
the second kind). Dyson and Lenard made the initial breakthrough
in 1967 for a non-relativistic model, and since then the problem
has attracted a lot of attention, various relativistic models
having been intensively studied in recent years. As was
demonstrated by Fr\"{o}hlich, Lieb and Loss in \cite{FLL}, to
establish stability, it is of crucial importance to know if the
kinetic energy operator has zero modes, i.e. eigenvectors
corresponding to an eigenvalue at 0; the possibility that zero
modes can exist was established at about the same time by Loss and
Yau in \cite{LY}. Subsequently, Balinsky and Evans showed in
\cite{BA} that zero modes are rare for these problems.

The first objective of the research reported here was to confirm
that the results in \cite{BA} can be extended to Dirac-type
operators with matrix-valued potentials. This then set the scene
for the main goal which was to determine the decay rates of zero
modes whenever they occur. The Dirac operator considered is of the
form,
\begin{equation}\label{Dir}
    H= \ba \cdot \p+ Q, \ \ \ \p = -i \bgrad
\end{equation}
where $\ba$ is the triple of Dirac matrices and $Q$ is a $4 \times
4$ matrix-valued function. Assuming that
$\|Q(\cdot)\|_{\mathbb{C}^4} \in L^3(\R^3),$ where
$\|\cdot\|_{\mathbb{C}^4}$ denotes any matrix norm on
$\mathbb{C}^4,$ it was shown that $Q$ is a small perturbation of
$\ba \cdot \p$ and hence (\ref{Dir}) defines a self-adjoint
operator $H$ as an operator sum with domain
$\left(H^{1,2}(\R^3)\right)^4,$ the space of $4$-component spinors
in $\left(L^2(\R^3)\right)^4$ with weak first derivatives in
$\left(L^2(\R^3)\right)^4$. The technique in \cite{BA} readily
applied to (\ref{Dir}) to meet the first objective and yield the
following result:

\begin{Theorem} Let $\|Q(\cdot)\|_{\mathbb{C}^4} \in L^3(\R^3).$ Then $H_t :=\ba \cdot \p+
tQ, t\in R^+$ can have a zero mode for only a countable set of
values of $t.$ Moreover
\[
\rm{nul} (H) \le const. \int_{\R^3} |Q(\x)|^3 d\x,
\]
where $\rm{nul}(H)$ denotes the nullity of $H$, i.e. the dimension
of the kernel of $H.$
\end{Theorem}

It is proved in \cite{SU} (see also \cite{SU2}) that if
$\|Q(\x)\|_{\mathbb{C}^4} = O(<\x>^{-\rho}), \rho >1,$ where
$<\x>:= \sqrt{1+|\x|^2},$ then any zero mode is $O(|\x|^{-2})$ at
$\infty $ and there are no resonances (defined as a solution
$\psi$ say which is such that $<\x>^{-s} \psi \in
\left(L^2(\R^3)\right)^4$ for some $s>0$) if $ \rho
> 3/2:$ note that such a $Q$ satisfies the assumption of Theorem 1.
In view of this result we concentrate our attention on more
singular potentials $Q,$ namely ones which satisfy
\begin{equation}\label{As}
    \|Q(\x)\|_{\mathbb{C}^4} = O(|\x|^{-1}), \ \ \ \x \in B_1^c,
\end{equation}
where $B_1$ is the unit ball, centre the origin in $\R^3,$ and
$B_1^c$ denotes its complement. We actually consider weak
solutions of $H \psi =0$ in $\left(L^2(B_1^c)\right)^4.$ Hence the
behaviour of $Q$ at the origin, and indeed within $B_1,$ does not
feature. Our main result is given in Theorem 4 below. Our approach
is very different to that in \cite{SU} and is based on two
techniques. In the first we use inversion with respect to $B_1$ to
replace the problem in $B_1^c$ by an analogous one in $B_1.$ The
second part involves establishing Sobolev-type embedding theorems
for the spaces $\mathbb{H}^{1,p}(\Omega)$ defined as the
completion of $\left[C_0^{\infty}(\Omega)\right]^4$ with respect
to the norm
\begin{equation}\label{2.1}
    \|f\|_{1,p;\Omega} := \left \{ \int_{\Omega} (|(\ba \cdot \p)f|^p +
    |f|^p)  d\x \right \}^{1/p}.
\end{equation}
For $p=2$ these are just vector versions of the standard Sobolev
spaces, but we need cases $ p \neq 2 $ for which we were unable to
find any appropriate results in the literature. The results we
give in section 3 use a method of Ledoux in \cite{L} to derive a
weak inequality and then embedding properties of Lorentz spaces on
bounded domains.

\section{Reduction by inversion}

We recall that in (\ref{Dir}), $\bf{\alpha} =
(\alpha_1,\alpha_2,\alpha_3)$ are Hermitian $4\times4$ matrices
satisfying
\[
\alpha_j\alpha_k + \alpha_k\alpha_j =2 \delta_{jk} I_4,\ \ \ j,k =
1,2,3,
\]
where $I_4$ is the unit $ 4 \times 4$ matrix, and we choose
\begin{equation}\label{1.1a}
\alpha_j = \left(\begin{array}{ll} 0_2 & \sigma_j \\
\sigma_j & 0_2
\end{array} \right),\ \ \ j=1,2,3,
\end{equation} where the $\sigma_j$ are the Pauli matrices
\begin{equation}\label{1.1b}
   \sigma_1 = \left(\begin{array}{ll} 0 & 1 \\
1 & 0
\end{array} \right),\ \ \ \sigma_2 = \left(\begin{array}{ll} 0 & -i \\
i & 0 \end{array}\right),\ \ \ \sigma_3 = \left(\begin{array}{ll} 1 & 0 \\
0 & -1 \end{array}\right)
\end{equation}
and $0_2$ is the $2 \times 2$ zero matrix. From (\ref{As}) we have
that the $4 \times 4$ matrix-valued function $Q$ has components
$q_{jk}, j,k =1,\cdots,4$ which satisfy
\[
|q_{jk}(\x)| \le C |\x|^{-1},\ \ \ |\x| \ge 1.
\]

For $ \mathbb{C}^4$-valued functions $f, g,$ measurable on an open
set $ \Omega \subseteq \R^3 $, and $z,w \in \mathbb{C}^4,$ we
shall use the following notation:
\begin{eqnarray*}
  <z,w> &:=& \sum _{j=1}^4z_j \overline{w_j},\ \ \ |z| := <z,z>^{\frac12} \\
  (f,g)_{\Omega} &:=& \int_{\Omega} <f(\x), g(\x)> d\x \\
  \|f\|_{p,\Omega} &:=& \left(\int_{\Omega} |f(x)|^p d\x \right)^{1/p}.
\end{eqnarray*}
Thus $ (f,g)_{\Omega}$ is the standard inner-product on
$\left(L^2(\Omega)\right)^4 $ and $\|f\|_{\Omega} =
\|f\|_{2,\Omega} $ the standard norm: when $\Omega = \R^3$ we
shall simply write $(f,g)$ and $\|f\|$.

Let $\psi$ be a weak solution of $ H \psi =0$ in $ \R^3 \setminus
\overline{B}_1$, where $B_1$ is the open unit ball centre the
origin and $ \overline{B}_1$ is its closure: hence for all $\phi
\in \left(\Con(\R^3 \setminus \overline{B}_1)\right)^4$
\begin{equation}\label{weak1}
    I:= \int <H \psi(\x),\phi(\x)> d\x =0.
\end{equation}
Inversion with respect to $B_1$ is the involution $ Inv: \x
\mapsto \y, \y= \x/ |\x|^2,$ and for any function defined on $\R^3
\setminus B_1$, the map $M : \phi \mapsto \tilde{\phi}:= \phi
\circ Inv^{-1}$ is such that $ \tilde{\phi}(\y) = \phi(\x) $ and
yields a function on $B_1$. Hence $\phi \in \left(\Con(\R^3
\setminus \overline{B}_1)\right)^4 $ means that $\tilde{\phi} \in
\left(\Con(B_1\setminus \{0\})\right)^4.$ The inversion gives
\begin{equation}\label{inversion}
 M\{ (\ba \cdot \p )\psi\}(\y) = |\y|^2 ({\bf{\beta}}\cdot \p)
 \tilde{\psi}(\y),
\end{equation}
where $ {\bf{\beta}} = (\beta_1,\beta_2,\beta_3)$ and
\[
\beta_k(\y) = \sum_{j=1}^3 \ba_j\left( \delta^k_j -
\frac{2y_ky_j}{|\y|^2}\right),
\]
where $\delta_j^k$ is the Kronecker delta function. It is readily
verified that the matrices $\beta_k(\y)$ are Hermitian and satisfy
\[
\beta_k(\y)\beta_j(\y) + \beta_j(\y) \beta_k(\y) = 2\delta_j^k
I_4.
\]
Also there exists a unitary matrix $ X(\y)$ such that $ X \in
C^{\infty}(\R^3\setminus \{0\})$ and for all $\y \neq 0,$
\begin{equation}\label{diagonal}
    X(\y)^{-1} \beta_k(\y) X(\y) = -\ba_k, \ \ \ k=1,2,3.
\end{equation}
Setting $\omega := \y/|\y|$ it is easy to verify that these
conditions are satisfied by
\begin{equation}\label{MatrixX}
   X(\y)= \left(\begin{array}{ll}X_2(\y) & O_2 \\
O_2 & X_2(\y) \\
\end{array} \right),
\end{equation}
where
\begin{equation*}
    X(\y) = \left( \begin{array}{ll} i\omega_3 & \omega_2
    +i\omega_1 \\
    -\omega_2+i \omega_1 & -i \omega_3 \\
    \end{array} \right).
\end{equation*}
Let $ \tilde{\psi}(\y) = -X(\y) \Psi(\y).$ Then from
(\ref{inversion}) we have
\begin{equation}\label{transsp}
M\{ (\ba \cdot \p )\psi\}(\y) = |\y|^2 X(\y)\left\{
(\ba\cdot\p)\Psi(\y) + Y(\y) \Psi(\y)\right \},
\end{equation}
where
\begin{equation}\label{Y}
    Y(\y) = \sum_{k=1}^3
    \alpha_k X(y)^{-1}\left(-i\frac{\partial}{\partial_{y_k}}X(\y)\right).
\end{equation}
Also, a calculation gives that the Jacobian of the inversion gives
$d\x = |\y|^{-6} d\y.$

Returning now to (\ref{weak1}), and with $\tilde{\phi}(\y) =
-X(\y) \Phi(\y)$ , the inversion yields
\begin{eqnarray*}
    I&=& \int_{B_1}< |\y|^2 X(\y)\left\{
(\ba\cdot\p)\Psi(\y) + Y(\y) \Psi(\y)\right. \\
& -& \left. \tilde{Q}(\y)X(\y)\Psi(\y)\right \}, X(\y)\Phi(\y)>
|\y|^{-6} d\y =0
\end{eqnarray*}
for all $\Phi \in \left(\Con(B_1\setminus \{0\})\right)^4$, which
can be written as
\begin{equation}\label{invweak1}
    I = \int_{B_1}<
(\ba\cdot\p)\Psi(\y) + Z(\y) \Psi(\y), |\y|^{-4}\Phi(\y)>  d\y =0
\end{equation}
where
\begin{equation}\label{Z}
    Z(\y) = Y(\y) - |\y|^{-2} X(\y)^{-1}\tilde{Q}(\y) X(\y).
\end{equation}
Equivalently, we can remove the factor $|\y|^{-4}$ in
(\ref{invweak1}) to give
\begin{equation}\label{invweak2}
    I = \int_{B_1}<
(\ba\cdot\p)\Psi(\y) + Z(\y) \Psi(\y), \Phi(\y)>  d\y =0
\end{equation}
for all $ \Phi \in \left(\Con(B_1\setminus \{0\})\right)^4.$

Let $\zeta \in C^{\infty}(\R^+)$ satisfy
\[
\zeta(t) = \left \{ \begin{array} {ll} 0 & {\rm{for}}\ \ 0<t<1 \\
1 & {\rm{for}} \ \ t>2
\end{array} \right.
\]
and for $\y \in \R^3 $ set $\zeta_n(\y)= \zeta(n|\y|).$ Then
\[
\grad \zeta_n(\y) =n \zeta'(n|\y|) \frac{\y}{|\y|}
\]
and so
\[
|\grad \zeta_n(\y)| = O\left(n
\chi_{[\frac1n,\frac2n]}(|\y|)\right)
\]
where $\chi_I$ denotes the characteristic function of the interval
$I.$ We then have from (\ref{invweak2}), now for all $\Phi \in
\left(C_0^{\infty}(B_1)\right)^4,$
\begin{eqnarray}\label{invweak3}
    I &=& \int_{B_1}<
(\ba\cdot\p)\Psi(\y) + Z(\y) \Psi(\y),
\zeta_n(\y)\Phi(\y)>  d\y \nonumber \\
&=& \int_{B_1}< \zeta_n(\y)\left\{(\ba\cdot\p)\Psi(\y) + Z(\y)
\Psi(\y)\right \}, \Phi(\y)> d\y \nonumber \\
&=& 0.
\end{eqnarray}
In $I$,
\[
\zeta_n(\y)(\ba\cdot\p)\Psi(\y) = (\ba\cdot\p)(\zeta_n \Psi)(\y)
-[(\ba\cdot\p)\zeta_n]( \Psi)(\y)
\]
and
\[
V_n(\y) :=[(\ba\cdot\p)\zeta_n] = O(n \chi_{[\frac1n,\frac2n]}).
\]
Therefore, as $ n \rightarrow \infty,$
\begin{eqnarray*}
  |\int_{B_1}<V_n(\y)\Phi(\y), \Phi(\y)> d\y | &\le& n \|\Psi\|_{L^{\infty}(B_1)}\int_{B_1} \chi_{[\frac1n,\frac2n]}d\y \\
   &=& O(n^{-2}) \rightarrow 0.
\end{eqnarray*}
Also
\begin{align*}
&\int_{B_1}< [(\ba\cdot\p)+Z(\y)] \zeta_n(\y)\Psi(\y),
\Phi(\y)> d\y \\
& = \int_{B_1}<  \zeta_n(\y)\Psi(\y),
[(\ba\cdot\p)+Z(\y)]\Phi(\y)> d\y \\
& \rightarrow \int_{B_1}<  \Psi(\y), [(\ba\cdot\p)+Z(\y)]\Phi(\y)>
d\y.
\end{align*}
We have therefore proved that
\begin{equation}\label{invweakfinal}
\int_{B_1}< [(\ba\cdot\p)+Z(\y)] \Psi(\y), \Phi(\y)> d\y =0
\end{equation}
for all $\Phi \in \left( \Con(B_1)\right)^4.$ In other words
\begin{equation}\label{invweakfinal2}
[(\ba\cdot\p)+Z(\y)] \Psi(\y) = 0
\end{equation}
in the weak sense. From (\ref{Z}) it follows that
\begin{equation}\label{boundonZ}
    \|Z(\y)\|_{\mathbb{C}^4} \le C |\y|^{-1}.
\end{equation}

\bigskip

\section{Dirac-Sobolev inequalities}

Let $ \mathbb{H}^{1,p}(\Omega), 1\le p <\infty, $ denote the
completion of $[\Con(\Omega)]^4$ with respect to the norm
\begin{equation}\label{2.1}
    \|f\|_{1,p;\Omega} := \left \{ \int_{\Omega} (|(\ba \cdot \p)f|^p +
    |f|^p)  d\x \right \}^{1/p}.
\end{equation}
We also use the notation
\begin{equation}\label{DandP}
\D := (\ba \cdot \p)^2, \ \ \  \mathbb{P}_t : e^{- t\D},\ \ t \ge
0.
\end{equation}
Then,
\[
(\mathbb{D}f)_j = -\Delta f_j,\ \ \ (\mathbb{P}_t f)_j = e^{-t
\Delta}f_j,\ \ j=1,2,3,4,
\]
where $\{e^{-t \Delta}\}_{t\ge 0}$ is the heat semigroup.
Furthermore, for all $t>0, \x \in \Omega,$
\begin{equation}\label{2.2}
    \left(e^{-t\Delta} f_j\right)(\x) = \frac{1}{(4\pi t)^{3/2}} \int_{\R^3}
    f_j(\y) e^{-|\x-\y|^2/4t} d\y,
\end{equation}
and so
\begin{equation}\label{2.2A}
\left(\mathbb{P}_t f\right)(\x) = \frac{1}{(4\pi t)^{3/2}}
\int_{\R^3}
    f(\y) e^{-|\x-\y|^2/4t} d\y.
\end{equation}
Note that if $\Omega \neq \R^n $ we put any $f\in
\mathbb{H}^{1,p}(\Omega)$ to be zero outside $\Omega$ and hence is
in $\mathbb{H}^{1,p} \equiv \mathbb{H}^{1,p}(\R^n).$ Define
\begin{equation}\label{2.3}
    \|f\|_{B^{\alpha}(\Omega)} :=  \sup_{t>0} \left \{ t^{-\alpha/2}
    |\mathbb{P}_t f|_{\infty; \Omega} \right \},
\end{equation}
where $ |\mathbb{P}_t f|_{\infty; \Omega} := \sup_{\x \in \Omega}
| \mathbb{P}_t f(\x)|,$ and denote by $B^{\alpha}(\Omega)$ the
completion of $\left[\Con(\Omega)\right]^4$ with respect to $
\|\cdot\|_{B^{\alpha}(\Omega)}.$

Our main theorem in this section introduces the weak-$L^q$ space
on $\Omega$, written, $L^{q,\infty}(\Omega),$ which is defined by
\[
\|f\|_{q,\infty;\Omega} : \sup_{u>0} \left \{u^q \lambda ( |f| \ge
u )\right \} ,
\]
where $\lambda $ denotes Lebesgue measure and $\lambda ( |f| \ge u
)$ stands for the measure of the set in $\Omega$ on which
$|f(\x)|\ge u.$

\begin{Theorem}\label{embthm}
Let $1\le p< q <\infty$ and let $f $ be such that $\|(\ba \cdot
\p)f\|_{p,\Omega}< \infty $ and $ f \in
B^{\theta/(\theta-1)}(\Omega),$ for $ \theta = p/q.$ Then we have
for some constant $C >0,$
\begin{equation}\label{DSineq}
    \|f\|_{q,\infty; \Omega} \le C \|(\ba \cdot \p)f\|_{p,\Omega}^{\theta}
    \|f\|^{1-\theta}_{B^{\theta/(\theta-1)}(\Omega)}.
\end{equation}
\end{Theorem}
\begin{proof}
For simplicity of notation, we suppress the dependence of the
norms and spaces on $\Omega$ throughout the proof. It is
sufficient to prove the result for $f\in
\left(C_0^{\infty}(\Omega)\right)^4.$ The proof is inspired by
that of Ledoux in \cite{L}. By homogeneity, we may assume that
$\|f\|_{B^{\theta/(\theta-1)}} \le 1,$ and so
\[
| \mathbb{P}_tf|_{\infty} \le t^{\theta/2(\theta-1)}
\]
for all $t>0.$ Thus, on choosing $t_u = u^{2(\theta-1)/\theta}$ it
follows that
\begin{equation}\label{hom}
| \mathbb{P}_{t_u}f|_{\infty} \le u.
\end{equation}
This gives that $|f| \ge 2u $ implies that $ |f-
\mathbb{P}_{t_u}f| \ge |f| - |\mathbb{P}_tf| \ge u $ and
consequently
\begin{eqnarray}\label{prePoin}
  u^q \lambda (|f| \ge 2u) &\le & u^q \lambda(|f- \mathbb{P}_{t_u}| \ge u) \nonumber \\
   & \le & u^{q-p} \int |f- \mathbb{P}_{t_u}|^p d \x.
\end{eqnarray}

Since
\[
\frac{\partial}{\partial t} \mathbb{P}_t f = (\ba \cdot \p)^2
\mathbb{P}_t f, \ \ \ \mathbb{P}_0 f = f,
\]
in view of the analogous result for $e^{-t\Delta}$ on each
component of $f$, it follows that
\[
\mathbb{P}_tf -f = \int_0^t (\ba \cdot \p)^2 \mathbb{P}_s f ds
\]
and hence for all $g \in \left[\Con(B_1)\right]^4,$
\begin{eqnarray*}
  \int_{\R^3} <g, f-\mathbb{P}_t f>d\x &=& -\int_0^t \left( \int_{\R^3}<g, (\ba \cdot \p)^2 \mathbb{P}_s f> d\x\right)ds \\
   &=& - \int_0^t \left(\int_{\R^3} < (\ba \cdot \p) \mathbb{P}_s g, (\ba \cdot
   \p)f> d\x \right)ds \\
   & \le & \|(\ba \cdot \p)f\|_p \int_0^t \|(\ba \cdot \p)
   \mathbb{P}_s g \|_{p'}ds,
\end{eqnarray*}
where $p'=p/(p-1)$ for $p>1 $ and $p'=\infty$ otherwise. From
(\ref{2.3}) we have
\begin{eqnarray*} (\ba \cdot \p)
   \mathbb{P}_s g(\x) &=& \frac{1}{(4\pi s)^{3/2}}\sum_{j=1}^3 \ba_j \int_{\R^3} g(\y)
   (-i \frac{\partial}{\partial x_j})e^{-\frac{|\x-\y|^2}{4s}}d\y \\
   &=& \frac{i}{(4\pi s)^{3/2}}\frac{1}{2s} \int_{\R^3}e^{-\frac{|\x-\y|^2}{4s}}
   [\ba \cdot (\x-\y)]g(\y)d\y. \\
\end{eqnarray*}
On using Young's inequality for convolutions, this yields
\begin{eqnarray}\label{2.8}
\|(\ba \cdot \p)
   \mathbb{P}_s g \|_{p'} &\le & \frac{1}{(4\pi s)^{3/2}}\frac{1}{2s}\int_{\R^3}
   [\ba \cdot \z]e^{-\frac{|\z|^2}{4s}}d\z \ \ \|g\|_{p'}\nonumber \\
&\le & C s^{-\frac12}\ \ \|g\|_{p'}
\end{eqnarray} for all $ p \in [1,\infty).$ We therefore have
\[
|\int <g, f-\mathbb{P}_t f>d\x| \le Ct^{\frac12}\|(\ba \cdot
\p)f\|_p \|g\|_{p'}
\]
and thus
\begin{equation*}
\|f-\mathbb{P}_t f\|_p \le C t^{\frac12}\|(\ba \cdot \p)f\|_p.
\end{equation*}
On substituting this in (\ref{prePoin}) we have
\begin{eqnarray*}
u^q \lambda(|f|\ge 2u) &\le& C u^{q-p}t_u^{p/2}\|(\ba \cdot
\p)f\|_p^p \\
&=& C \|(\ba \cdot \p)f\|_p^p
\end{eqnarray*}
since $q-p+p(\theta -1)/\theta = 0,$ whence the result.

\end{proof}

In the following corollary, the notation indicates that
integration is over $B_1.$

\begin{Corollary}\label{embd2}
Let $1\le p<q <\infty, r:= 3(\frac{q}{p}-1) \in [1, p]$ and $f \in
\mathbb{H}^{1,p}(B_1).$ Then we have that for any $k \in (0,q)$
and $\theta = p/q,$ there exists a positive constant $C$ such that
\begin{equation}\label{2.9}
    \|f\|_{k,B_1} \le C \|(\ba \cdot
    \p)f\|^{\theta}_{p,B_1}\|f\|_{r,B_1}^{1-\theta}
\end{equation}
\end{Corollary}
\begin{proof}
All norms in the proof are over $B_1.$ From (\ref{2.3}), for any
$r\in [1,\infty]$ and with $r'=r/(r-1), r>1,$
\begin{eqnarray*}
  |\mathbb{P}_t f(\x)|  &\le& \frac{1}{(4\pi t)^{\frac{3}{2}}}\|f\|_r \left(\int_{B_1}e^{-r'|\x-\y|^2/4t}d\y \right)^{1/r'} \\
   &\le & C t^{-3/2r}\|f\|_r;
\end{eqnarray*}
note that this holds also for $r=1.$ Hence,
\begin{eqnarray*}
  \|f\|_{B^{\theta/(\theta-1)}} &\le & C \sup_{t>0} t^{-\frac{\theta}{2(\theta-1)}-\frac{3}{2r}}\|f\|_r \\
   &=& C \|f\|_r
\end{eqnarray*}
if $\frac{\theta}{2(\theta-1)}+\frac{3}{2r} =0,$ which is true if
$r=3(q/p -1)$ since $\theta =p/q.$ Since $r \le p$ and $f \in
L^p(B_1)$ it follows that $f \in L^r(B_1)$ and hence $f \in
B^{\theta/(\theta-1)}.$ From (\ref{DSineq}) we therefore have
\begin{equation} \label{2.10}
\|f\|_{q,\infty} \le C \|(\ba \cdot \p)f \|_p^{\theta} \|f\|_r^{(1
-\theta )}.
\end{equation}
But for Lorentz spaces $L^{r,s}$ on a set $\Omega$ of finite
measure, we have the continuous embeddings (see \cite{EE2},
Proposition 3.4.4)
\begin{equation}\label{Lor}
    L^{q,s}(\Omega) \hookrightarrow L^{k,m}
\end{equation}
if $ 0< k<q \le \infty , 0 <s,m \le \infty.$ In particular, with
$s=\infty, m=k,$ and recalling that $L^{k,k} = L^k,$ we have for
$0<k<q \le \infty,$
\[
L^{q,\infty} \hookrightarrow L^k
\]
and so
\[
\|f\|_k \le C\|f\|_{q,\infty}.
\]
The corollary follows from (\ref{2.10}).
\end{proof}

\begin{Corollary}\label{embd2A}
Let $p \in [1,\infty), k\in [1,p(p+3)/3)$ and $f \in
\mathbb{H}^{1,p}(B_1).$ Then there exists a positive constant $C$
such that
\begin{equation}\label{Sob3}
    \|f\|_{k,B_1} \le C \|(\ba \cdot \p)f\|_{p, B_1}.
\end{equation}
\end{Corollary}
\begin{proof}
For $k \in [1,p],$ we choose in Corollary 1, $q=p(k+3)/3$ and
$r=k$ to deduce (\ref{Sob3}) from (\ref{2.9}). When $k
\in(p,4p/3)$ we choose any $q \in[4p/3, p(p+3)/3],$ so that
$q>k>p$ and $r=3(q/p -1) \in [1,p].$ When $k \in[4p/3,p(p+3)/3)$
we choose any $q \in (k,p(p+3)/3]$ so that $q>k>p$ and $r\in
(1,p].$ In both the last two cases we also have $ r< k$ and hence
\[
\|f\|_{r,B_1} \le C \|f\|_{k,B_1}.
\]
This yields (\ref{Sob3}) from (\ref{2.9}).
\end{proof}

\section{Estimate for zero modes}

Let $\psi$ be such that $ \psi, (\ba \cdot \p)\psi \in L^2(B_1^c),
B_1^c := \R^3 \setminus B_1,$ and
\begin{equation}\label{3.0}
(\ba \cdot \p)\psi
(\x) = -Q(\x) \psi(\x),
\end{equation}
where
\[
 \|Q(\x)\|_{\mathbb{C}^4} = O(|\x|^{-1}),
\]
and
\[
\psi \in \left(L^2(B_1^c)\right)^4.
\]
If $\theta \in C^1(\R^3)$ is $1$ in the neighbourhood of $\infty$
and is supported in $\R^3 \setminus \overline{B_1}$ then $\theta
\psi $ has similar properties to those above. Hence we may assume,
without loss of generality, that $\psi$ is supported in $\R^3
\setminus \overline{B_1}.$ Moreover,
\begin{equation}\label{3.1}
\int_{B_1^c}|\x|^2 |(\ba \cdot \p)\psi(\x)|^2 d\x < \infty.
\end{equation}
On applying the inversion described in section 1, and using the
notation
\[
 (M \psi)(\y) =:
\tilde{\psi}(\y) =: -X(\y) \Psi(\y)
\]
where $ |X(\y)| \asymp 1,$ we have from $\psi \in L^2(\B_1^c)$
that
\begin{equation}\label{3.2}
 \int_{B_1}|\Psi(\y)|^2 \frac{d\y}{|\y|^6}< \infty
\end{equation}
We also have from (\ref{transsp}) that
\begin{equation}\label{3.3}
\left(  M(\ba\cdot \p)\psi\right)(\y) = |\y|^2
\left\{X(\y)\left[(\ba\cdot\p)\Psi(\y) + Y(\y)
   \Psi(\y)\right]\right\},
\end{equation}
where $Y(\y)$ is given in (\ref{Y}) and is readily seen to satisfy
\[
\|Y(\y)\|_{\mathbb{C}^4} \asymp 1/|\y|.
\]

Let $\Psi(\y) = |\y|^2\Phi(\y).$ Then
\[
(\ba\cdot \p)\Psi(\y) = |\y|^2\left\{(\ba\cdot \p)\Phi(\y) +
O\left(\frac{|\Phi(\y)|}{|\y|}\right)\right\}
\]
and so from (\ref{3.3})
\begin{equation}\label{3.4}
 | M(\ba\cdot \p)\psi(\y)|  \asymp  |\y|^4\left\{(\ba\cdot \p)\Phi(\y)+ O\left(\frac{|\Phi(\y)|}{|\y|}\right) \right
   \}.
\end{equation}
Hence from (\ref{3.0})
\begin{equation}\label{3.5}
    \int_{B_1} |\y|^{-2}\left||\y|^4\left\{(\ba\cdot \p)\Phi(\y)+ O\left(\frac{|\Phi(\y)|}{|\y|}\right) \right
   \}\right|^2\frac{d\y}{|\y|^6} < \infty.
\end{equation}
Since, from (4.3),
\[
\int_{B_1} |\Phi(\y)|^2 \frac{d\y}{|\y|^2} < \infty
\]
it follows from (\ref{3.5}) that
\[
\int_{B_1} |(\ba \cdot \p)\Phi(\y)|^2 dy < \infty.
\]
Hence, $ \Phi \in \mathbb{H}^{1,2}(B_1):$ recall that we may
assume that $\psi$ is supported in $\R^3 \setminus \overline{B_1}$
and hence $\Phi$ is supported in $B_1.$ By Corollary 2, we have
that $\Phi \in L^k(B_1)$ for any $ k \in[1,10/3).$

 We therefore have the preliminary result

\begin{Theorem} \label{est1}
Let $\psi \in L^2(B_1^c)$ be a solution of (\ref{3.0}) with $Q(\x)
=O(|\x|^{-1})$ in $B_1^c.$ Then, for any $k \in [1,10/3),$
\[
\psi(\x) = |\x|^{-2} \phi(\x),
\]
where
\[
\int_{B_1^c}|\phi(\x)|^k |\x|^{-6}d\x < \infty.
\]

\end{Theorem}

However, we can repeat the analysis that led to Theorem 3 to
improve this result. From $\Psi(\y)=|\y|^2\Phi(\y)$ and (2.14) it
follows that
\begin{equation}\label{repeat}
    (\ba \cdot \p)\Phi(\y) = -Z^{(1)}(\y)\Phi(\y),
\end{equation}
where
\[
Z^{(1)}= Z(\y) - 2i|\y|^{-2}(\ba \cdot \y) = O(\frac{1}{|\y|}).
\]
Let $\Phi(\y) =|\y|^t \Phi^{(2)}(\y).$ Then we have
\[
(\ba \cdot \p)\Phi^{(2)}(\y) = -Z^{(2)}(\y)\Phi^{(2)}(\y),
\]
where
\[
Z^{(2)}(\y) = O(\frac{1}{|\y|}).
\]
By H\"{o}lder's inequality, for $1 \le p < k, k \in[1,10/3),$
\begin{eqnarray*}
  \int_{B_1} \left|Z^{(2)}(\y)\Phi^{(2)}(\y)\right|^p d\y & \le & C  \int_{B_1} \left||\y|^{-(1+t)}\Phi(\y)\right|^p d\y \\
   & \le & \left(\int_{B_1}|\Phi(\y)|^k d\y  \right)^{p/k}
   \left(\int_{B_1}|\y|^{-p(1+t)(k/(k-p)} d\y   \right)^{1-p/k} \\
   & < & \infty
\end{eqnarray*}
if
\begin{equation}\label{exponent}
    p \left(\frac{1+t}{3} + \frac{1}{k}\right) < 1.
\end{equation}
Choose $p=1$ and let $ 0<t<11/10.$ Then (\ref{exponent}) is
satisfied for some $ k\in [1,10/3),$ depending on $t.$ We
therefore have that $\Phi^{(2)} \in \mathbb{H}^{1,1}(B_1)$ if
$0<t<11/10,$ and hence, by Corollary 2, that $\Phi^{(2)} \in
L^s(B_1)$ for any $s \in [1,4/3).$ The following result has
consequently been proved.

\begin{Theorem}
Suppose that $\psi \in L^2(B_1^c)$ satisfies $ \{ (\ba \cdot \p) +
Q(\x)\} \psi(\x) =0 $ in $B_1^c.$ Then for any $0< t< 11/10, $
$\psi(\x) = |\x|^{-2-t}\phi(\x),$ where
\[
\int_{B_1^c}|\phi(\x)|^s|\x|^{-6} d\x < \infty,
\]
for any $s \in [1,4/3).$

\end{Theorem}

\bibliographystyle{amsalpha}

\begin{thebibliography}{22}
\bibitem{BA} {A.~Balinsky and W.~D.~Evans,} On the zero modes of Weyl-Dirac operators and their
multiplicity. {\it{Bull.L.M.S.}} {\bf{34}} (2002), 236-242.
\bibitem{EE2}{D.~E.~Edmunds and W.~D.~Evans,}{\it Hardy Operators,
Function Spaces and Embeddings.}{Springer: Berlin, Heidelberg, New
York, 2004}.
\bibitem{FLL}{J.~Fr\"{o}lich, E.~Lieb and M.~Loss,} Stability of
Coulomb systems with magnetic fields, I: the one-electron atom.
{\it{Comm. Math. Phys.}} {\bf{104}} (1986) 251-270.
\bibitem{L}{M.~Ledoux,} On improved Sobolev embedding theorems,
{\it Math.Res. Lett.} {\bf 10}, 659-669 (2003).
\bibitem{LY} {M.~Loss and H-T Yau,} Stability of Coulomb systems
with magnetic fields, III: zero energy bound states of the Pauli
operator. {\it{Comm. Math. Phys.}} {\bf{104}} (1986) 283-290.
\bibitem{SU} {Y.~Sait\={o} and T.~Umeda,} The zero modes and
zero resonances of massless Dirac operators. To appear in
{\it{Hokkaido Math.J.}.}
\bibitem{SU2} {Y.~Sait\={o} and T.~Umeda,} The asymptotic limits of
zero modes of massless Dirac operators. To appear in {\it{Letters
of Math.Phys.}}
\end{thebibliography}

\end{document}